\documentclass[a4paper,12pt]{amsart}
\usepackage{amssymb}
\usepackage{amsmath, amsthm, amscd, amsfonts, amssymb, graphicx, color}
\usepackage{amsmath, amsthm, amscd, amsfonts, amssymb, graphicx, color}

\usepackage[cp1250]{inputenc}
\usepackage[T1]{fontenc}

\newtheorem{theorem}{Theorem}[section]
\newtheorem{lemma}[theorem]{Lemma}
\newtheorem{proposition}[theorem]{Proposition}
\newtheorem{corollary}[theorem]{Corollary}
\theoremstyle{definition}
\newtheorem{definition}[theorem]{Definition}

\theoremstyle{remark}
\newtheorem{remark}[theorem]{Remark}
\numberwithin{equation}{section}

\begin{document}
	\author{Stefan  Ivkovi\'{c}}

	\vspace{15pt}
	
	\title{Dynamics of operators on the space of Radon measures}

	\maketitle
	
		\begin{center}
			Mathematical Institute of the Serbian Academy of Sciences and Arts,  Kneza Mihaila 36,  11000 Beograd,Serbia
		\end{center}
	\begin{abstract}
		In this paper, we study the dynamics of the adjoint of a weighted composition operator and we give necessary and sufficient conditions for this adjoint operator to be topologically hyper-transitive on the space of Radon measures on a locally compact Hausdorff space. Moreover, we provide sufficient conditions for this operator to be chaotic and we give concrete examples. Next, we consider the real Banach space of signed Radon measures and we give in this context the sufficient conditions for the convergence of Markov chains induced by the adjoint of an integral operator. Also, we illustrate this result by a concrete example. In addition, we obtain some structural results regarding the space of Radon measures. We characterize a class of cones whose complement is spaceable in the space of Radon measures.
	\end{abstract}
	
	\vspace{15pt}
	
	\begin{flushleft}
	\end{flushleft} 
	
	\vspace{15pt}
	
	\begin{flushleft}
	\end{flushleft}
	
	\vspace{30pt}

\section{Introduction}
\label{sec:1}
Linear dynamics of operators have been studied in many articles during several decades; see \cite{gpbook} and \cite{bmbook} as monographs on this topic.
Among other concepts, hypercyclicity, topological transitivity and topological mixing, as important linear dynamical properties of bounded linear operators, have been investigated in many research works; see \cite{an97, bg99, kuchen17, shaw, IT, Filomat, kostic} and their references.
Specially, hypercyclic weighted shifts on $\ell^p(\mathbb{Z})$ were characterized in \cite{sa95,ge00}, and then  C-C. Chen and 
C-H. Chu, using aperiodic elements of locally compact groups, extended the results in \cite{sa95} to weighted translations on Lebesgue spaces in the context of a second countable group \cite{cc11}. Afterwards, many other linear dynamics in connection with this theme have been studied; see \cite{chen11, chen141, IT2}. Recently, the dynamics of weighted composition operators on solid Banach function spaces has been studied in \cite{ tsi, si09, sts} .

Now, for a locally compact Hausdorff space $\Omega ,$ we can consider $C_{0}(\Omega)$  as a special kind of a solid Banach function space. Here, as usual, $C_{0}(\Omega)$ is the space of all continuous functions on $\Omega$ vanishing at infinity, which is equipped with the supremum norm.  It is well known in functional analysis that the dual of $C_{0}(\Omega) $ is isometrically isomorphic to the space of Radon measures on $\Omega$ equipped with the total variation norm. Hence, the adjoint of a weighted composition operator on $C_{0}(\Omega) $ can be considered as an operator on the space of Radon measures on $\Omega.$ In Section 3 of this paper we study the dynamics of this adjoint operator and provide conditions that ensure topological transitivity and chaos for this operator and for the cosine operator function generated by this operator. Moreover, we give concrete examples. 

Another essential topic in connection with the dynamics of operators is fixed point theory, which can in fact be viewed as an opposite aspect to hypercyclicity and topological transitivity. An important part of the fixed point theory deals with the convergence of Markov chains and many papers have been published on this theme, such as \cite{francuzi}. In Section 4 we consider a Markov  integral-operator that acts on the real Banach space of all real valued continuous functions on a compact Hausdorff space $ \Omega $ equipped with the supremum norm. The adjoint of this operator acts then on the real Banach space of all signed Radon measures on $ \Omega $ equipped with the total variation norm. Based on the main general result in \cite{francuzi}, we provide sufficient conditions on the integral kernel of this operator that ensure the convergence of the Markov chains induced by the ajoint of this operator. Also, we illustrate our result by a concrete example.

Now, the theory of Markov operators rely heavily on the concept of cones in Banach spaces. This fact motivated us to present some structural results concerning the cones in the space of Radon measures. Recall that a subset of a topological vector space is called spaceable 
if its union with the singleton $\{0\}$ contains a closed infinite-dimensional linear subspace.  This notion was introduced in \cite{fon,aro} and so far has been considered in several papers, such as \cite{MIA}. In Section 5,  by applying \cite[Theorem 4.2]{MIA},  we  characterize a class of cones whose complement is spaceable in the space of Radon measures on locally compact Hausdorff space $\Omega .$ As a corollary, we obtain that the complement of the cone consisting of all  positive Radon measures on $\Omega$ is spaceable when $ \Omega $ has infinite cardinality. 

\section{Preliminaries}
\label{sec:2}
If $\mathcal X$ is a Banach space, the set of all bounded linear operators from $\mathcal X$ into $\mathcal X$ is denoted by $B(\mathcal X)$. Also, we denote $\mathbb{N}_0:=\mathbb{N}\cup\{0\}$.\\

We recall now the following two definitions. The first definition is an extension of \cite[Definition 2.1]{tsi}.

\begin{definition} 
	Let $\mathcal X$ be a Banach space. A sequence $(T_n)_{n\in \mathbb{N}_0}$  of operators in $B(\mathcal X)$ is called {\it topologically transitive} if for each non-empty open subsets $U,V$ of
	${\mathcal X}$, $T_n(U)\cap V\neq \varnothing$ for some $n\in \mathbb{N}$.\\
	A sequence $(T_n)_{n\in \mathbb{N}_0}$  of operators in $B(\mathcal X)$ is called {\it topologically hyper-transitive} if for each non-empty open subsets $U,V$ of
	${\mathcal X}$, there exists a strictly increasing sequence  $\lbrace n_{k} \rbrace_{k} \subseteq \mathbb{N}$ such that  $T_{n_{k}}(U)\cap V\neq \varnothing$ for all $k\in \mathbb{N}$. If $T_n(U)\cap V\neq \varnothing$ holds from some $n$ onwards, then
	$(T_n)_{n\in \mathbb{N}_0}$ is called {\it topologically mixing}.\\
	A single operator $T$ in $B(\mathcal X)$ is called topologically transitive (respectively hyper-transitive or mixing) if the sequence $(T^n)_{n\in \mathbb{N}_0}$ is topologically transitive (respectively hyper-transitive or mixing). 
\end{definition}

\begin{definition} \cite[Definition 2.3]{tsi}
	Let $\mathcal X$ be a Banach space, and $(T_n)_{n\in \mathbb{N}_0}$  be a sequence of operators in $B(\mathcal X)$. A vector $x\in \mathcal X$ is called a {\it periodic element} of $(T_n)_{n\in \mathbb{N}_0}$ if there exists a constant $N\in\mathbb N$ such that for each $k\in\mathbb N$, $T_{kN}x=x$. The set of all periodic elements of $(T_n)_{n\in \mathbb{N}_0}$ is denoted by
	${\mathcal P}((T_n)_{n\in \mathbb{N}_0})$. The sequence $(T_n)_{n\in \mathbb{N}_0}$ is called {\it chaotic} if $(T_n)_{n\in \mathbb{N}_0}$ is topologically transitive and ${\mathcal P}((T_n)_{n\in \mathbb{N}_0})$ is dense in ${\mathcal X}$. An operator $T\in B(\mathcal{X})$ is called \emph{chaotic} if the sequence $\{T^n\}_{n\in \mathbb{N}_0}$ is chaotic. 
\end{definition}

The following concept plays a key role in the proofs.
\begin{definition} \cite[Definition 2.6]{tsi}
	Let $X$ be a topological space. Let $\alpha:X\longrightarrow X$ be invertible, and $\alpha,\alpha^{-1}$ be Borel measurable. We say that $\alpha$ is \emph{aperiodic} if for each compact subset $K$ of $X$, there exists a constant $N>0$ such that for each $n\geq N$, we have $K \cap \alpha^{n}(K)=\varnothing$, where $\alpha^{n}$ means the $n$-fold combination of $\alpha$. 
\end{definition}

Throughout this paper we let $\Omega $ be a locally compact Hausdorff space and $\alpha $ be an aperiodic homeomorphism of $\Omega $ in the case when $ \Omega$ is non-compact. As usual, $C_{0}(\Omega) $  denotes the space of all continuous functions on $\Omega$ vanishing at infinity, $C_{b}(\Omega)$ denotes the space of bounded, continuous functions on $\Omega$, whereas $C_{c}(\Omega)$ stands for the set of all continuous, compactly supported functions on $\Omega$. Both $C_{0}(\Omega) $ and $C_{b}(\Omega)$ are equipped with the supremum norm. Moreover, we let $w$ be a positive bounded function on $\Omega $ such that also $w^{-1} \in C_{b}(\Omega)$ and we put then $T_{\alpha,w} $ to be the weighted composition operator on $C_{0}(\Omega) $ with respect to $\alpha $ and $w ,$ that is $T_{\alpha , w} (f)=w \cdot (f \circ \alpha) $ for all $ f \in C_{0} (\Omega).$ Easily, one can see that by the above assumptions $T_{\alpha , w}$ is well-defined and $\|T_{\alpha,w}\|\leq \|w\|_{\sup}$.  
Since $\frac{1}{w}$ is also bounded, then $T_{\alpha , w}$ is invertible and we have 
\begin{equation*}
	T_{\alpha , w}^{-1}f=\frac{f\circ \alpha^{-1}}{w\circ \alpha^{-1}},\quad( f\in  C_{0}(\Omega)).
\end{equation*}
Simply we denote $S_{\alpha , w}:=T_{\alpha , w}^{-1}$.
\begin{remark} \cite[Remark 2.5]{tsi}
	If $w$ and $\frac{1}{w}$ are weights, the inverse of a weighted composition operator  $T_{\alpha,w}$ is also a weighted composition operator. In fact, $ S_{\alpha , w}=T_{\alpha^{-1}, \frac{1}{w \circ \alpha^{-1}}}$. Moreover, if $T_{\alpha_{1}, w_{1}} $ and $T_{\alpha_{2}, w_{2}}  $ are two weighted composition operators, then 
	$$T_{\alpha_{2}, w_{2}}\circ T_{\alpha_{1}, w_{1}} = T_{\alpha_{1}  \circ \alpha_{2} , w_{2}(w_{1} \circ \alpha_{2})},$$ so the composition of two weighted composition operators is again a weighted composition operator.
	By some calculation one can see that for each $n\in \mathbb{N}$ and $f\in C_{0}(\Omega)$,
	$$T^{n}_{\alpha,w}f =\left(\prod_{j=0}^{n-1}(w\circ \alpha^{j})\right) \cdot (f\circ \alpha^{n}) \text{    (1)}$$
	and 
	$$S^{n}_{\alpha,w}f =\left(\prod_{j=1}^{n}(w\circ \alpha^{-j})\right)^{-1}\cdot (f\circ \alpha^{-n}). \text{    (2)}$$
\end{remark}

The adjoint $T_{\alpha,w}^{*} $ is a bounded operator on $M(\Omega) $ where $M(\Omega) $ stands for the Banach space of all Radon measures on $\Omega $ equipped with the total variation norm. It is straightforward to check that $$T_{\alpha,w}^{*}(\mu)(E) = \int_{E} w \circ \alpha^{-1} \,d\mu\ \circ \alpha^{-1} $$ 
for every $\mu \in M(\Omega) $ and every measurable subset $E$ od $\Omega .$ Here $\mu\ \circ \alpha^{-1} (E) = \mu ( \alpha^{-1} (E)) $ for every  $\mu \in M(\Omega) $ and every measurable subset $E$ od $\Omega .$  By $(1)$ and $(2)$ it follows then that for every $n \in \mathbb{N} ,$ $\mu \in M(\Omega)  $ and a Borel measurable subset $E \subseteq \Omega $ we have that 
$${T_{\alpha,w}^{*n}}(\mu)(E) = \int_{E}  \prod_{j=0}^{n-1} w \circ \alpha^{j-n} \,d\mu\ \circ \alpha^{-n} $$ 
and 
$${S_{\alpha,w}^{*n}}(\mu)(E) = \int_{E} \prod_{j=1}^{n} (w \circ \alpha^{n-j})^{-1} \,d\mu\ \circ \alpha^{n}.$$ 

Throughout this paper, $M_+(\Omega)$ will denote the set of all non-negative Radon measures on $\Omega .$ For every Radon measure $ \mu $ on $ \Omega ,$ we let as usual $ \vert \mu \vert $ denote the total variation of $ \mu .$

\section{Topologically transitive operators on the space of Radon measures}
\label{sec:3}
We start with the following proposition. 

\begin{proposition} \label{r12 p3.1 }
	The following statements are equivalent.\\
	i) $T_{\alpha,w}^{*} $ is topologically hyper-transitive on $M(\Omega) .$\\
	ii) For every compact subset $K$ of $ \Omega$ and any two measures $\mu, v$ in $M(\Omega) $ with $\vert v \vert (K^{c})= \vert \mu \vert (K^{c})=0  $ there exist a strictly increasing sequence $\lbrace n_{k} \rbrace_{k} \subseteq \mathbb{N}  $ and sequences $ \lbrace A_{k} \rbrace,$ $\lbrace B_{k} \rbrace $ of Borel subsets of $K$ such that $\alpha^{n_{k}} (K) \cap K =\varnothing $ for all $ k \in \mathbb{N}$ and 
	$$\lim_{k \to \infty} \vert \mu \vert (A_{k} ) = \lim_{k \to \infty} \vert v \vert (B_{k} ) =0   ,$$
	$$ \lim_{k \to \infty} \sup_{t\in K \cap A_{k}^{c}} ( \prod_{j=0}^{n_{k}-1} (w \circ \alpha^{j})(t)) = \lim_{k \to \infty} \sup_{t\in K \cap B_{k}^{c}} ( \prod_{j=1}^{n_{k}} (w \circ \alpha^{-j})^{-1}(t)) = 0   .$$
\end{proposition}

\begin{proof}
	We prove $i) \Rightarrow ii)  $ first. Suppose that $T_{\alpha,w}^{*} $ is topologically hyper-transitive on $ M(\Omega).$ For a given compact subset $K$ of $\Omega,$ let $v, \mu \in M(\Omega) $ such that $\vert v \vert (K^{c})= \vert \mu \vert (K^{c})=0 .$ Since $ T_{\alpha,w}^{*} $ is topologically hyper-transitive, for each  $ k \in \mathbb{N}$ we can find some $\eta^{(k)} \in M(\Omega)  $ and a strictly increasing sequence $ \lbrace n_{k} \rbrace_{k \subseteq \mathbb{N}} $ such that $$\vert \eta^{(k)} -\mu \vert (\Omega) < \dfrac{1}{4^{k}} \textrm{ and } \vert \gamma^{(k)} -v \vert (\Omega) < \dfrac{1}{4^{k}} ,$$ where $\gamma^{(k)} = T_{\alpha,w}^{*n_{k}} (\eta^{(k)})  ,$ that is 
	$$\gamma^{(k)} (E) =  \int_{E} \prod_{j=0}^{n_{k}-1} (w \circ \alpha^{j-n_{k}} ) \,d \eta^{(k)} \ \circ \alpha^{-n_{k}}  $$ 
	for every measurable subset $E$ of $\Omega .$ Hence we must have that $ \vert \eta^{(k)} \vert (K^{c})  < \dfrac{1}{4^{k}}  $ and $ \vert \gamma^{(k)} \vert (K^{c})  < \dfrac{1}{4^{k}}  .$ Also we have that $ S_{\alpha,w}^{*n_{k}} (\gamma^{(k)}) = \eta^{(k)}  ,$ that is 
	$$  \int_{E} \prod_{j=1}^{n_{k}} (w \circ \alpha^{n_{k}-j} )^{-1} \,d \gamma^{(k)} \ \circ \alpha^{n_{k}} = \eta^{(k)}(E) $$ 
	for every measurable subset $E$ of $\Omega .$\\
	Now, since $\alpha $ is aperiodic, we may in fact choose the sequence $\lbrace n_{k} \rbrace  $ is a such way that $\alpha^{n_{k}} (K) \cap K =\varnothing $ and $\alpha^{-n_{k}} (K) \cap K =\varnothing $ for all $k.$
	
	If $\eta^{(k)} = \eta_{1,+}^{(k)} - \eta_{1,-}^{(k)} + i \eta_{2,+}^{(k)} -i \eta_{2,-}^{(k)} $ and $\gamma^{(k)} = \gamma_{1,+}^{(k)} -\gamma_{1,-}^{(k)} + i \gamma_{2,+}^{(k)} -i \gamma_{2,-}^{(k)} $ are Jordan decompositions of $\eta^{(k)} $ and $\gamma^{(k)} ,$ it follows that $( \eta_{1,+}^{(k)} + \eta_{1,-}^{(k)}) (K^{c})  < \dfrac{1}{4^{k}}   ,$ $ ( \eta_{2,+}^{(k)} + \eta_{2,-}^{(k)}) (K^{c})  < \dfrac{1}{4^{k}}  ,$ $( \gamma_{1,+}^{(k)} + \gamma_{1,-}^{(k)}) (K^{c})  < \dfrac{1}{4^{k}}    $  and $( \gamma_{2,+}^{(k)} + \gamma_{2,-}^{(k)}) (K^{c})  < \dfrac{1}{4^{k}}   $ Also, it is easy to see that 
	$$ \gamma_{j,\pm}^{(k)} (E) = \int_{E} \prod_{i=0}^{n_{k}-1} (w \circ \alpha^{i-n_{k}} ) \,d \ \eta_{j, \pm}^{(k)} \ \circ \alpha^{-n_{k}}  $$ 
	and 
	$$\eta_{j,\pm}^{(k)} (E) = \int_{E} \prod_{i=1}^{n_{k}} (w \circ \alpha^{n_{k}-i} )^{-1} \,d \ \gamma_{j, \pm}^{(k)} \ \circ \alpha^{n_{k}} $$ 
	for every measurable subset $E$ of $\Omega $ and $j \in \lbrace 1,2 \rbrace .$ This is because $ \alpha$ is a homeomorphism of $\Omega ,$ so taking compositions with $\alpha $ preserves mutual singularity of measures.  Hence, for $j \in \lbrace 1,2 \rbrace $ we get that 
	
	$$\int_{\alpha^{n_{k}} (K)} \prod_{i=0}^{n_{k}-1} (w \circ \alpha^{i-n_{k}} ) \,d \ ( \eta_{j, +}^{(k)}+ \eta_{j, -}^{(k)} )  \ \circ \alpha^{-n_{k}}  < \dfrac{1}{4^{k}}    ,$$ 
	and
	$$\int_{\alpha^{-n_{k}} (K)} \prod_{i=1}^{n_{k}} (w \circ \alpha^{n_{k}-i} )^{-1} \,d \ ( \gamma_{j, +}^{(k)}+ \gamma_{j, -}^{(k)} )  \ \circ \alpha^{n_{k}}  < \dfrac{1}{4^{k}}    .$$ 
	By the change of variables we obtain that 
	$$\int_{K} \prod_{i=0}^{n_{k}-1} (w \circ \alpha^{i} ) \,d \ ( \eta_{j, +}^{(k)}+ \eta_{j, -}^{(k)} )    < \dfrac{1}{4^{k}}    ,$$ 
	and 
	$$\int_{K} \prod_{i=1}^{n_{k}} (w \circ \alpha^{-i} )^{-1} \,d \ ( \gamma_{j, +}^{(k)}+ \gamma_{j, -}^{(k)} )   < \dfrac{1}{4^{k}}    .$$  
	for $j \in \lbrace 1,2 \rbrace .$ Let 
	$$A_{k}:= \lbrace t \in K \ \vert \ \prod_{i=0}^{n_{k}-1} (w \circ \alpha^{i} ) (t)  >  \dfrac{1}{2^{k}} \rbrace$$ 
	and 
	$$B_{k}:= \lbrace t \in K \ \vert \ \prod_{i=1}^{n_{k}} (w \circ \alpha^{-i} )^{-1}(t) >  \dfrac{1}{2^{k}} \rbrace .$$ 
	Then $  ( \eta_{j, +}^{(k)}+ \eta_{j, -}^{(k)} )  (A_{k})  < \dfrac{1}{2^{k}} $ and $ ( \gamma_{j, +}^{(k)}+ \gamma_{j, -}^{(k)} )  (B_{k})  < \dfrac{1}{2^{k}}  $ for $ j \in \lbrace 1,2 \rbrace.$ Hence $(\eta^{(k)}) (A_{k})  < \dfrac{1}{2^{k-1}} $ and $(\gamma^{(k)}) (B_{k})  < \dfrac{1}{2^{k-1}} .$ Therefore, $ \vert \mu \vert (A_{k}) < \dfrac{1}{2^{k-1}} + \dfrac{1}{4^{k}}  $ and $ \vert v \vert (B_{k}) < \dfrac{1}{2^{k-1}} + \dfrac{1}{4^{k}} .$ 
	
	Next we prove $ii) \Rightarrow i) .$ Given two open subsets $ O_{1}$ and $ O_{2} $ of $M(\Omega) ,$ choose some $\mu \in  O_{1} $ and $ v \in  O_{2}.$ By the regularity of $\mu $ and $v$ there exist some compact subsets $K_{1} $ and $K_{2} $ of $\Omega $ such that $ \vert \mu \vert (K_{1}^{c})< \dfrac{\epsilon}{2}$ and $ \vert v \vert (K_{2}^{c})< \dfrac{\epsilon}{2}$ for sufficiently small $ \epsilon$ satisfying that $\epsilon -$neighbourhoods of $\mu $ and $v$ are contained in $O_{1} $ and $O_{2} ,$ respectively. Set $K=K_{1} \cup K_{2} .$ Then $K$ is compact and $\vert \mu \vert (K^{c}),\vert v \vert (K^{c}) < \dfrac{\epsilon}{2} .$Therefore, we may without loss of generality, assume that there exists some compact subset $ K \subseteq \Omega$ such that $\vert \mu \vert (K^{c})=\vert v \vert (K^{c}) =0 .$
	
	Choose the strictly increasing sequence $ \lbrace n_{k} \rbrace_{k}$ and the sequences $ \lbrace A_{k} \rbrace_{k} ,  \lbrace B_{k} \rbrace_{k} $ of Borel subsets of $K$ satisfying the assumptions in $ii)$ with respect to $\mu $ and $v.$ For each $k \in \mathbb{N} ,$ let 
	$\tilde{\mu}_{k}, \tilde{v}_{k} $ be the measures in $M(\Omega) $ given by $\tilde{\mu}_{k}(E)=\mu (E \cap A_{k}^{c}  \cap K )$ and $\tilde{v}_{k} (E)=v (E \cap B_{k}^{c}  \cap K ) $ for every measurable subset $E$ of $\Omega .$ Then 
	$$T_{\alpha,w}^{*n_{k}} ( \tilde{\mu}_{k} ) (E)  =  \int_{E} \prod_{j=0}^{n_{k}-1} (w \circ \alpha^{j-n_{k}} ) \,d \tilde{\mu}_{k} \ \circ \alpha^{-n_{k}}$$ 
	for every measurable subset $E$ of $\Omega .$	However, 
	$$ \int_{E} \prod_{j=0}^{n_{k}-1} (w \circ \alpha^{j-n_{k}} ) \,d \tilde{\mu}_{k} \ \circ \alpha^{-n_{k}} = 
	\int_{  E \cap \alpha^{n_{k}} (A_{k}^{c} \cap K )  } \prod_{j=0}^{n_{k}-1} (w \circ \alpha^{j-n_{k}} ) \,d \tilde{\mu}_{k} \ \circ \alpha^{-n_{k}}
	=$$
	$$\int_{  \alpha^{-n_{k}}   (E) \cap A_{k}^{c} \cap K  } \prod_{j=0}^{n_{k}-1} (w \circ \alpha^{j} ) \,d \tilde{\mu}_{k}, 
	$$ 
	so we get that 
	$$ \vert   T_{\alpha,w}^{*n_{k}} ( \tilde{\mu}_{k} ) (E) \vert \leq \sup_{t \in A_{k}^{c} \cap K}     \prod_{j=0}^{n_{k}-1} (w \circ \alpha^{j} )(t)  \vert \tilde{\mu}_{k}  \vert ( \alpha^{-n_{k}}   (E) \cap A_{k}^{c} \cap K ) = $$ 
	$$
	\sup_{t \in A_{k}^{c} \cap K}     \prod_{j=0}^{n_{k}-1} (w \circ \alpha^{j} )(t)  \vert \mu  \vert ( \alpha^{-n_{k}}   (E) \cap A_{k}^{c} \cap K )
	$$
	for every measurable subset $E$ of $\Omega .$ It follows that 
	$$ 
	\parallel  T_{\alpha,w}^{*n_{k}} ( \tilde{\mu}_{k} ) \parallel  \leq \sup_{t \in A_{k}^{c} \cap K}   (  \prod_{j=0}^{n_{k}-1} (w \circ \alpha^{j} )(t) \ ) \parallel \mu  \parallel \text{    (3) }
	$$ for all $k,$ so $$\lim_{k \to \infty}  T_{\alpha,w}^{*n_{k}} ( \tilde{\mu}_{k} ) = 0 .$$ 
	Similarly, we can show that
	$$\parallel  S_{\alpha , w}^{*^{n_{k}}} (\tilde{v}_{k}) \parallel \leq 
	\sup_{t\in  B_{k}^{c} \cap K} 
	\text{ } ( \prod_{j=1}^{n_{k}} (w \circ \alpha^{-j})^{-1} (t) ) \parallel v \parallel  \text{  (4) }$$
	for all $k,$ so 
	$$ \lim_{k \to \infty}  S_{\alpha,w}^{*n_{k}} ( \tilde{v}_{k} ) = 0 .$$ Moreover, we have that $( \mu - \tilde{\mu}_{k}) (E) = \mu (E \cap A_{k})$ hence $\vert \mu - \tilde{\mu}_{k} \vert (E) \leq \vert \mu \vert (E \cap A_{k}) $ for all $k$ and every measurable subset $E$ of $ \Omega ,$ which gives $ \vert \mu - \tilde{\mu}_{k} \vert (\Omega) \leq \vert \mu \vert (A_{k}) $ for all $k.$ Thus, 
	$\lim_{k \to \infty} \parallel \mu - \tilde{\mu}_{k} \parallel = 0 .$ Similarly we can show that
	$$\lim_{k \to \infty} \parallel v - \tilde{v}_{k} \parallel = 0 .$$ 
	
	For each $ k \in \mathbb{N},$ put 
	$\eta_{k} = \tilde{\mu}_{k} +  S_{\alpha,w}^{*n_{k}} ( \tilde{v}_{k} ) .$ 
	Then we have that $\lim_{k \to \infty}  \eta_{k} = \mu $ and $\lim_{k \to \infty}   T_{\alpha,w}^{*n_{k}}  (\eta_{k})=v,$ which shows that $T_{\alpha,w}^{*} $ is topologically hyper-transitive.
\end{proof}

\begin{corollary} \label{r12 c1.2}
	We have that $ii) \Rightarrow i) $\\
	$i) $ $T_{\alpha,w}^{*} $is topologically hyper-transitive on  $M(\Omega) .$\\
	$ii)$ For every compact subset $K$ of $\Omega $ we have that 
	$$\lim_{n \to \infty} \ \sup_{t \in K} ( \prod_{j=0}^{n-1} (w \circ \alpha^{j} )(t)) =
	\lim_{n \to \infty} \ \sup_{t \in K} ( \prod_{j=1}^{n} (w \circ \alpha^{-j} )^{-1} (t)) =0  .$$ 
\end{corollary}

The cosine operator functions and the dynamics of such mappings have been studied in several papers, such as \cite{but,sha,bmi08,kalmes10}. In the rest of this section we shall consider the cosine operator functions generated by  $ T_{\alpha, W}^{*n}$ for each $n \in \mathbb{N} .$ Thus, for each $n \in \mathbb{N} ,$ we put $C_{\alpha, W}^{(n)^{*}} = \dfrac{1}{2} (T_{\alpha, W}^{*n} +  S_{\alpha, W}^{*n}) .$\\
The main idea for the proof of the next proposition is inspired by the proof of \cite[Theorem 5]{kalmes10}.

\begin{proposition} \label{r12 p3.4}
	We have that $(ii) \Rightarrow (i):$ \\
	$(i)$ The sequence $(C_{\alpha,w}^{(n)^{*}}) $ is topologically hyper-transitive on $M(\Omega) .$ \\
	$(ii)$ For every compact subset $K $ od $\Omega $ and any two measures $\mu,v$ in $M(\Omega)$ with $\vert \mu \vert (K^{c}) = \vert v \vert (K^{c}) =  0  $ there exist a strictly increasing sequence $\lbrace n_{k} \rbrace \subseteq \mathbb{N} $ and sequences $\lbrace A_{k} \rbrace_{k} , \lbrace F_{k} \rbrace_{k}  , \lbrace D_{k} \rbrace_{k}  $ of Borel subsets of $K$ such that 
	$$ \lim_{k \rightarrow \infty } \vert \mu \vert (A_{k}) = \lim_{n \rightarrow \infty } \vert v \vert (A_{k}) = 0  ,$$
	
	$$\lim_{k \rightarrow \infty } \sup_{t\in K \cap A_{k}^{c}} 
	\text{ } ( \prod_{j=0}^{n_{k}-1} (w \circ \alpha^{j}) (t) ) =
	\lim_{k \rightarrow \infty } \sup_{t\in K \cap A_{k}^{c}} 
	\text{ } ( \prod_{j=0}^{n_{k}-1} (w \circ \alpha^{-j})^{-1} (t) ) = 0
	,$$ 
	
	$$
	\lim_{k \rightarrow \infty } \sup_{t\in F_{k}} 
	\text{ } ( \prod_{j=0}^{2 n_{k}-1} (w \circ \alpha^{j}) (t) ) =
	\lim_{k \rightarrow \infty } \sup_{t\in D_{k}} 
	\text{ } ( \prod_{j=1}^{2 n_{k}} (w \circ \alpha^{-j})^{-1} (t) ) = 0,
	$$
	
	where $F_{k} \cap D_{k} = \varnothing $ and $A_{k}^{c} \cap K = F_{k} \cup D_{k}$ for all $k.$
\end{proposition}

\begin{proof}
	As in the proof of Proposition \ref{r12 p3.1 }, given two open subsets $O_{1} $ and $O_{2} $ of $M(\Omega) ,$ there exists some compact subset $K$ of $\Omega $ and measures $ \mu \in O_{1}, v \in O_{2}$ such that $ \vert \mu \vert (K^{c}) = \vert v \vert (K^{c}) =  0.$ Choose the strictly increasing sequence $ \lbrace n_{k} \rbrace_{k}$ and the sequences of Borel subsets of $K, \lbrace A_{k} \rbrace_{k} , \lbrace F_{k} \rbrace_{k}  , \lbrace D_{k} \rbrace_{k} $ satisfying the assumptions of the proposition with respect to $\mu,v$ and $K.$ For each $ k \in \mathbb{N},$ let $\tilde{\mu}_{k}, \tilde{v}_{k} $ and $\tilde{ \tilde{v}}_k$ be the measures in $M(\Omega)$ defined by 	
	$  \tilde{v}_{k} (E)= v (E \cap F_{k}), \tilde{ \tilde{v}}_k (E) = v (E \cap D_{k}) $ and $ \tilde{\mu}_{k} (E)= \mu ( E \cap A_{k}^{c} \cap K )$ 
	for every measurable subset $E$ of $ \Omega .$ Then 
	$( \tilde{v}_{k} +  \tilde{ \tilde{v}}_{k}) (E) = v (E \cap A_{k}^{c} \cap K) $ for every measurable subset $E$ of $\Omega $ and it is not hard to see that $\parallel v -  \tilde{v}_{k} -  \tilde{ \tilde{v}}_{k} \parallel \stackrel{k \rightarrow \infty }{\longrightarrow} 0$ since $ \vert v \vert (A_{k}) \rightarrow 0 $ as $k \rightarrow \infty $ by the assumption. Similarly, since $ \vert \mu \vert (A_{k}) \rightarrow 0 $ as $k \rightarrow \infty ,$ we have $ \parallel \mu - \mu_{k} \parallel \rightarrow 0 $ as $ k \rightarrow \infty.$
	
	By $(3)$ we have that 
	$$\parallel T_{\alpha , w}^{*^{n_{k}}} (\tilde{v}_{k}) \parallel \leq 
	\sup_{t\in F_{k}} 
	\text{ } ( \prod_{j=0}^{n_{k}-1} (w \circ \alpha^{j}) (t) ) \parallel v \parallel \leq 
	\sup_{t\in A_{k}^{c} \cap K} 
	\text{ } ( \prod_{j=0}^{n_{k}-1} (w \circ \alpha^{j}) (t) ) \parallel v \parallel 
	\stackrel{k \rightarrow \infty }{\longrightarrow} 0
	.$$ 
	Similarly, we can show that
	
	$$\lim_{k \rightarrow \infty } T_{\alpha , w}^{*^{2n_{k}}} (\tilde{v}_{k})=0 , \lim_{k \rightarrow \infty } S_{\alpha , w}^{*^{n_{k}}} (\tilde{ \tilde{v}}_{k})=0,  \lim_{k \rightarrow \infty } S_{\alpha , w}^{*^{2n_{k}}} (\tilde{ \tilde{v}}_{k})=0 ,$$
	$$ \lim_{k \rightarrow \infty } T_{\alpha , w}^{*^{n_{k}}} (\tilde{\mu}_{k})=0 \text{ and} \lim_{k \rightarrow \infty } S_{\alpha , w}^{*^{n_{k}}} (\tilde{\mu}_{k})=0 .$$ 
	Set for each $ k \in \mathbb{N},$ 
	$\eta_{k}=\tilde{\mu}_{k} +2 T_{\alpha , w}^{*^{n_{k}}} (\tilde{v}_{k}) + 
	2S_{\alpha , w}^{*^{n_{k}}} (\tilde{ \tilde{v}}_{k})  $ and proceed as in the proof of Proposition \ref{r12 p3.1 } part $ii) \Rightarrow i) .$
\end{proof}

\begin{corollary} \label{r12 c3.4}
	We have that $(ii) \Rightarrow (i):$ \\
	$(i)$ The sequence $(C_{\alpha,w}^{(n)^{*}}) $ is topologically hyper-transitive on $M(\Omega) .$ \\
	$(ii)$ For every compact subset $K $ od $\Omega $ we have that 
	
	$$\lim_{n \rightarrow \infty } \sup_{t\in K } \text{ } ( \prod_{j=0}^{n-1} (w \circ \alpha^{j}) (t) ) =
	\lim_{n \rightarrow \infty } \sup_{t\in K } 
	\text{ } ( \prod_{j=0}^{n-1} (w \circ \alpha^{-j})^{-1} (t) ) = 0
	,$$ 
	
	$$
	\lim_{n \rightarrow \infty } \sup_{t\in K } 
	\text{ } ( \prod_{j=0}^{2 n-1} (w \circ \alpha^{j}) (t) ) =
	\lim_{n \rightarrow \infty } \sup_{t\in K } 
	\text{ } ( \prod_{j=1}^{2 n} (w \circ \alpha^{-j})^{-1} (t) ) = 0
	.$$
\end{corollary}

In the next proposition we give sufficient conditions for the operator $T_{\alpha , w}^{*} $ and for the cosine operator function generated by this operator to be chaotic. The main idea in the proof of this proposition is inspired by the proof of \cite[Theorem 5.3]{kalmes11}.

\begin{proposition}
	We have that $ii) \Rightarrow i) .$ \\
	$(i)$ The sequences $(T_{\alpha , w}^{*n})_{n} $ and  $(C_{\alpha , w}^{(n)^{*}})_{n} $ are chaotic on $M(\Omega) .$\\
	$(ii)$ For any compact subset $K$ of $\Omega $ and any measure $ \mu \in M(\Omega)$ with $\vert \mu \vert (K^{c}) = 0 ,$ there exist a sequence of Borel subsets $ \lbrace A_{k} \rbrace_{k} $ of $K$ and a strictly increasing sequence $\lbrace n_{k} \rbrace_{k} \subseteq \mathbb{N} $ such that $\lim_{k \rightarrow \infty } \vert \mu \vert (A_{k}^{c}) = 0 $ and 
	$$\lim_{k \rightarrow \infty } \sum_{l=1}^{\infty} \sup_{t\in K \cap A_{k}^{c}} 
	\text{ } ( \prod_{j=0}^{ln_{k}-1} (w \circ \alpha^{j}) (t) ) = 
	\lim_{k \rightarrow \infty } \sum_{l=1}^{\infty} \sup_{t\in K \cap A_{k}^{c}} 
	\text{ } ( \prod_{j=1}^{ln_{k}} (w \circ \alpha^{-j})^{-1} (t) ) =0   ,$$ 
	where the corresponding series are convergent for each $k.$ 
\end{proposition}

\begin{proof}
	By Proposition \ref{r12 p3.4} it suffices to show that $\mathcal{P}((T_{\alpha , w}^{*n})_{n})$ and $\mathcal{P}((C_{\alpha , w}^{(n)^{*}})_{n}) $ are dense in $M(\Omega) .$ Let $ O $ be a non-empty open subset in $M(\Omega) .$ Then, as observed earlier, there exists some measure $\mu \in O $ and a compact subset $K$ of $\Omega $ such that $\vert \mu \vert (K^{c})=0 .$ Choose sequence $ \lbrace A_{k} \rbrace $ of Borel subsets and a strictly increasing sequence $ \lbrace n_{k} \rbrace_{k} \subseteq \mathbb{N}$ satisfying the assumptions in $(ii)$ with respect to $ \mu $ and $K.$ For each $ k \in \mathbb{N} ,$ let $\mu_{k}$ be the measure given by $\mu_{k} (E) = \mu (E \cap A_{k}^{c} \cap K) $ for every measurable subset $E$ of $\Omega.$ Then $ \parallel \mu - \mu_{k} \parallel \rightarrow 0 $ as $ k \rightarrow \infty$ since $\lim_{k \rightarrow \infty } \vert \mu \vert (A_{k}) = 0 .$ Moreover, by $(3),$ $(4)$ and the assumptions in $(ii)$ we have that the series $  \sum_{l=1}^{\infty} T_{\alpha , w}^{*^{ln_{k}}} (\mu_{k}), \sum_{l=1}^{\infty} S_{\alpha , w}^{*^{ln_{k}}} (\mu_{k})  $ are absolutely convergent for each $k$ and moreover 
	$$\lim_{k \rightarrow \infty } \sum_{l=1}^{\infty} \parallel T_{\alpha , w}^{*^{ln_{k}}} (\mu_{k}) \parallel =  
	\lim_{k \rightarrow \infty } \sum_{l=1}^{\infty} \parallel S_{\alpha , w}^{*^{ln_{k}}} (\mu_{k}) \parallel = 0.$$ 
	For each $k \in \mathbb{N} ,$ set 
	$$v_{k}= \mu_{k} + \sum_{l=1}^{\infty} T_{\alpha , w}^{*^{ln_{k}}} (\mu_{k}) + \sum_{l=1}^{\infty} S_{\alpha , w}^{*^{ln_{k}}} (\mu_{k})
	.$$ 
	It is straightforward to check that $v_{k}=T_{\alpha,w}^{*ln_{k}} v_{k}=C_{\alpha,w}^{(ln_{k})^{*}} v_{k}$ for all $l, k \in \mathbb{N}.$ Moreover, for each $k \in \mathbb{N},$ by the relations (3) and (4) we have  
	$$\parallel v_{k} - \mu \parallel \leq \parallel \mu_{k} - \mu \parallel + \sum_{l=1}^{\infty} \parallel T_{\alpha , w}^{*^{ln_{k}}} (\mu_{k}) \parallel + \sum_{l=1}^{\infty} \parallel S_{\alpha , w}^{*^{ln_{k}}} (\mu_{k}) \parallel $$
	$$\leq \parallel \mu_{k} - \mu \parallel + \sum_{l=1}^{\infty} \sup_{t \in A_{k}^{c} \cap K} ( \prod_{j=0}^{ln_{k}-1} (w \circ \alpha^{j}) (t) ) \parallel \mu \parallel $$ $$ +\sum_{l=1}^{\infty} \sup_{t \in A_{k}^{c} \cap K} ( \prod_{j=0}^{ln_{k}} (w \circ \alpha^{-j})^{-1} (t) ) \parallel \mu \parallel .$$
	By the assumptions in \textit{(ii)} it follows that $ \lim_{k \rightarrow \infty } v_{k}=\mu,$ which proves the proposition.	
\end{proof}

\begin{corollary} \label{r12 c3.6}
	We have that $ii) \Rightarrow i) .$ \\
	$(i)$ The sequences $(T_{\alpha , w}^{*n})_{n} $ and  $(C_{\alpha , w}^{(n)^{*}})_{n} $ are chaotic on $M(\Omega) .$\\
	$(ii)$ For any compact subset $K$ of $\Omega $ we have that 
	$$\lim_{n \rightarrow \infty } \sum_{l=1}^{\infty} \sup_{t\in K } 
	\text{ } ( \prod_{j=0}^{ln-1} (w \circ \alpha^{j}) (t) ) = 
	\lim_{n \rightarrow \infty } \sum_{l=1}^{\infty} \sup_{t\in K } 
	\text{ } ( \prod_{j=1}^{ln} (w \circ \alpha^{-j})^{-1} (t) ) =0   ,$$ 
	where the corresponding series are convergent for each $n.$ 
\end{corollary}

As a concrete example, let $\Omega = \mathbb{R}, \ \alpha (t) = t+1 $ for all $ t \in \mathbb{R} $ and 
$$w(t)=
\begin{cases}
	2 \ \text{ for } t \leq -1, & \\
	\dfrac{1}{2} \ \text{ for } t \geq 1 ,&  \\
	\text{linear on the segment } [-1,1] .& 
\end{cases}$$

In this case the conditions of Corollary \ref{r12 c1.2}, Corollary \ref{r12 c3.4} and  Corollary \ref{r12 c3.6} are satisfied. 

In general, if $M, \epsilon > 0 $ such that $1+ \epsilon < M ,$ $1 - \epsilon > \dfrac{1}{M} ,$ and $K_{1}, K_{2} > 0 ,$ then if $w \in C_{b} (\mathbb{R}) $ satisfies that $M \geq \vert w(t) \vert \geq 1 + \epsilon $ for all $t \leq -K_{1} $ and $\dfrac{1}{M} \leq \vert w(t) \vert \leq 1 - \epsilon $ for all $t \geq K_{2} ,$ then the conditions of Corollary \ref{r12 c1.2}, Corollary \ref{r12 c3.4} and  Corollary \ref{r12 c3.6} are satisfied.\\
In fact, it is not hard to see from the proof of Proposition \ref{r12 p3.1 } part $ii) \Rightarrow i) $ and from the proof of Proposition \ref{r12 p3.4} that the sufficient conditions of Corollary \ref{r12 c1.2} and Corollary \ref{r12 c3.4} ensure also that the respective operators will be even topologically mixing and not just topologically hyper-transitive. 
\section{Convergence of Markov chains in the space of signed Radon measures}

In this section, we shall first recall the notions and definitions from \cite{francuzi} and \cite{master} that are needed to prove the main result denoted as Proposition \ref{glavna} .

\text{ } 

Let $X$ be a real vector space and $ K\subseteq X, \ K\neq \emptyset.$ We say that $K$ is a cone in $X$ if $K$ satisfies the following properties:
\begin{enumerate}
	\item $K+K\subseteq K .$
	\item $\lambda K \subseteq K $ for all $\lambda >0 .$
	\item $K \cap (-K)= \lbrace 0 \rbrace .$
\end{enumerate}

We let $\leq $ denote the associated order on X, so $x\leq y $ if and only if $ (y-x) \in K$. For $e \in K$, set $$I_e= \lbrace x\in X \ | \ -e \leq x \leq e \rbrace .$$ Recall that $ A \subset X $ is called absorbing if $$\lbrace t>0 \ | \ x \in tA \rbrace \neq \emptyset $$ for every $x \in X$. An element $e \in K $ is called an order unit for $(X,K) $ if $I_e$ is absorbing. 
\begin{flushleft}
	If $\mathrm{ Int } K \neq \emptyset $ for  $x \in X$ and $y \in \mathrm{Int} (K)$ we define then   
\end{flushleft}
\begin{center}
	$M(x/y):=\inf \lbrace t \in \mathbb{R} \ : \ x\leq ty \rbrace $
\end{center}
\begin{center}
	and $m(x/y):=\sup \lbrace t \in \mathbb{R} \ : \ x\geq ty \rbrace .$
\end{center}
\begin{flushleft}
	For an order unit $e \in \mathrm{K}$, we define the Thompson's norm with respect to $e$ to be given by $||x||_T:=\max \lbrace M(x/e),-m(x/e)\rbrace $ for all $x \in X .$ 
\end{flushleft}

Let $e \in K \backslash \lbrace 0 \rbrace.$ Then $e$ is an order unit if and only if $e$ is an internal point of $K ,$ that is for all  $ x \in X $ there exists a $ \delta >0 $ such that $(e+ \lambda x ) \in K $ for all $ \lambda \in [-\delta, \delta ] .$  

For $e \in K\setminus \lbrace 0\rbrace $, set $$X_e= \bigcup_{t \geq 0 } tI_e \ , \ K_e=K\cap X_e.$$ 
Then $X_e $ is a subspace of $X$, $I_e$ is an absorbing, absolutely convex subset of $X_e$, $K_e$ is a cone of $X_e$ and $e$ is an order unit for $(X_e,K_e)$. If $|\cdot|_e $ denotes the Minkowski functional on $X_e$ associated with $I_e$, (so $|x|_e=\inf \lbrace t>0 \ | \ x \in tI_e \rbrace$), then $|\cdot|_e$ is a seminorm on $X_e$ such that $|e|_e=|-e|_e=1.$ All these facts can be easily checked by elementary calculations. For the details, we refer to \cite[Section 2]{master}.\\

Next we recall the following definition. 

\begin{definition} \cite{francuzi} \cite[Definition 1.1]{master} Let $(X,||\cdot ||)$  be a real normed space, let $K$ be a cone in $X$. We say that $K$ is normal if there is some constant $M>0 $ such that $||x||\leq M||y|| $ for all $x,y \in K $ satisfying $x \leq y .$
\end{definition}
We recall also the following useful proposition.
\begin{proposition} \label{master} \cite[Proposition 1.2]{master}
	
	Assume that $(X,||\cdot ||)$ is a real normed space, K is closed and normal cone in $X$ and let $e \in K\setminus \lbrace 0 \rbrace$. Then the following statements hold.
	\begin{enumerate}
		\item $|\cdot |_e$ is a norm on $X_e$ and $||\cdot|| \leq \alpha |\cdot |_e$ on $X_e $ for some $\alpha >0 .$
		\item If $(X,||\cdot||)$ is a Banach space, then $(X_e,|\cdot |_e) $ is a Banach space.
		\item If $e$ is an order-unit in $(X,K)$, then $X=X_e$.
		\item If $e \in \mathrm{Int}(K)$, where $\mathrm{Int}(K)$ is the interior of $K$ in topological sense, then $e$ is order-unit in $(X,K)$ and $||\cdot ||$ , $|\cdot|_e$ are equivalent norms on $X$. 
	\end{enumerate} 
	Moreover, $|\cdot |_e$ is the same as the Thomson's norm on $X$ defined above.  
\end{proposition}

From now on in this section, unless else is specified, we let $(X,||\cdot ||) $ be a real Banach space, $K \subseteq X $ be a normal, closed cone with Int$(K)\neq \emptyset , \ e\in \mathrm{ Int }K$ be an order unit and $||\cdot ||_T$ be the Thompson' norm with respect to $e .$
Since $||\cdot||$ and $||\cdot ||_T$ are equivalent because K is closed and normal, we have  $ (X,||\cdot ||)^*=(X,||\cdot ||_T)^*$. This space will be denoted by $X^*$ from now on.\\
Furthermore, we define the dual cone $K^*$ in $X^*$ by $$K^*= \lbrace x \in X^* \ | \ z(x)  \geq 0 \ \forall x \in K \rbrace $$ and the abstract simplex by $$ P(e)=\lbrace \mu \in K^* \ | \ \mu (e) \ =1 \rbrace .$$

Given $(X, e, ||\cdot ||_T)$, consider now the quotient space $X /{ \mathbb{R} e} $ and define Hilbert's quotient norm $||| \cdot |||_H : X/{\mathbb{R} e} \rightarrow \mathbb{R}^+$ by $$|||x+\mathbb{R} e|||_H=2\inf_{\lambda \in \mathbb{R}} ||x+\lambda e ||_T .$$ Then $|||\cdot |||_H$  is a norm on $X / {\mathbb{R} e} .$

Moreover, let $M(e)=\lbrace \mu \in X^* | \ \mu (e) \ =0 \rbrace $ and define Hilbert's dual norm  
$$||\cdot ||^*_H :M(e) \rightarrow \mathbb{R}^+ \text{ } \mathrm{by} \text{ } ||\mu ||_H^*=\frac{1}{2}||\mu||_T^* \ \ \textrm{for all } \mu \in M(e). $$ 
We have then the following lemma.
\begin{lemma} \cite{francuzi} \cite[Lemma 4.2]{master} $(X /{\mathbb{R} e}^*, |||\cdot |||_H^* ) $ is isometrically isomorphic to  \\
	$(M(e),||\cdot ||_H^* ).$
\end{lemma}
We define now the induced linear map  $\tilde{T}:X/\mathbb{R}e \rightarrow X/\mathbb{R}e $ by $$\tilde{T}(x+\mathbb{R}e )=T(x)+\mathbb{R}e .$$
It is straightforward to check that $\tilde{T}$ is a well defined, bounded linear map with respect to $|||\cdot|||_H.$
Hence we may consider $|||\tilde{T}|||_H$, that is the operator norm of $\tilde{T}$ with respect to $|||\cdot|||_H.$\\
We also define $S^*:M(e)\rightarrow M(e)$ by letting $S^*=T_{|_{M(e)}} .$ One can show that $||S^*||_H^*=|||\tilde{T}|||_H ,$ for more details, see \cite[Section 6]{master}.

We recall that a bounded linear operator $T$ on $X$ is said to be a Markov operator with respect to $K$ and $e$ if $T(K) \subseteq K$ and $T(e) = e .$

\text{ }

The crucial tool for proving the main result in this section is the following theorem.
\begin{theorem} \label{theomaster} \cite{francuzi} \cite[Theorem 6.1]{master}
	Let $ T: X \rightarrow X $ be a Markov operator with respect to $K$ and $e$. If $\vert \vert \vert \tilde{T} \vert \vert \vert_H <1 ,$ or equivalently if $\parallel S^* \parallel_H^* <1 ,$ then there is $ \pi \in P(e) $ such that for all $x \in X$ and $n \in \mathbb{N},$ we have 
	$$ \parallel T^n (x) - \pi (x) e\parallel_T \leq (\vert \vert \vert \tilde{T} \vert \vert \vert_H)^n \parallel x \parallel_H $$ 
	$$ \mathrm{and} \ \parallel (T^*)^n ( \mu  ) -\pi \parallel_H^* \leq (\vert \vert \vert \tilde{T} \vert \vert \vert_H)^n$$ $ \ for \ all \ \mu \in P(e) $  \textit{and all } $n\in \mathbb{N}.$
\end{theorem}

\text{ }

Now we let $\Omega $ be a compact Hausdorff space and we set $X= C_\mathbb{R} ( \Omega )$, that is X is the space of all continuous real valued functions on $ \Omega $. $X$ is then a Banach space with the supremum norm $ \parallel \cdot \parallel_\infty .$ 
We put
$$ K= \lbrace f \in X \vert \text{ } f(w) \geq 0\ \forall w \in \Omega \rbrace $$
and we let the order unit $e$ be the constant function 1 on $\Omega . $ It is obvious that $K $ closed and normal cone with respect to $||\cdot ||_\infty $ and that $ 1 \in \mathrm{ Int }(K) .$

\text{ }

We have the following useful auxiliary lemma, which has also been proved in \cite[Example 1.6]{master} .

\begin{lemma} 
	For all $ f \in X ,$ we have 
	$$\parallel f \parallel_T = \parallel f \parallel_\infty .$$
\end{lemma}

\begin{proof}
	It holds that 
	$$ \parallel f \parallel_T= \inf \lbrace t>0 \vert\ f \in tI_e \rbrace $$
	since $\parallel \cdot \parallel_T$ is equal to the Minkowski functional with respect to $I_e $ by Proposition \ref{master}. Hence we obtain that
	$$ \parallel f \parallel_T = \inf \lbrace t>0 \vert\ f \in tI_e \rbrace ,$$
	$$=\inf \lbrace t>0 \vert\ -t\leq f(w) \leq t\ \forall w \in \Omega \rbrace ,$$
	$$=\inf \lbrace t>0 \vert\ \vert f(w)\vert \leq t\ \forall w \in \Omega \rbrace = \parallel f \parallel_\infty ,$$ which proves the lemma.
\end{proof} 

Let $M_r (\Omega )$ denote the space of all signed Radon measures  on $\Omega $ with the norm $ \parallel v \parallel = \vert v \vert ( \Omega ).$
For $v \in M_r(\Omega ), $ let $ \phi_v : C_\mathbb{R} (\Omega ) \rightarrow \mathbb{R}$ be defined by $$\phi_v (f) =\int_\Omega f dv.$$
\begin{flushleft}
	Then $v\rightarrow \phi_v $ is an isometric isomorphism of $ (M_r,( \Omega ), \parallel \cdot \parallel )$ onto  $(( C_\mathbb{R} (\Omega ))^{*}, \parallel \cdot \parallel ) .$ For the proof, see \cite[Remark 2.3]{master}.
	\begin{flushleft}
		Since $ (C_{\mathbb{R} } (\Omega ), ||\cdot ||_\infty  )^* $ is isometrically isomorphic to $M_r(\Omega )$ equipped with the total variation norm, it follows that $$ ||\mu||_T^* = |\mu |(\Omega ) \ \textrm{for all } \mu \in M_r (\Omega ).$$
	\end{flushleft}
	Also, it is clear that the dual cone and simplex in this case are $K^*=M_+(\Omega),$ 
	$$P(1)=\lbrace \mu \in M_+(\Omega) \ | \ \mu(\Omega)=1 \rbrace, $$
	so the simplex is the set of all probability Radon measures on $\Omega$.
\end{flushleft}

\text{ }

The material in the rest of this section contains the unpublished results from \cite[Section 9]{master} .\\

Let $ k: \Omega \times \Omega \rightarrow \mathbb{R} $ be a continuous non - negative function and let $\mu $ be a positive  Radon measure on $\Omega $ such that
$$\int_\Omega k(x,y) d \mu (y) >0 \ \ \textrm{for all } x \in \Omega .$$

\begin{flushleft}
	Put then $ \tilde{k}: \Omega * \Omega \rightarrow \mathbb{R} $  to be defined as 
\end{flushleft}
$$\tilde{k}(x,y)=\frac{k(x,y)}{\int_\Omega k(x,y) d \mu (y)} .$$
\begin{flushleft}
	Then $\tilde{k}$ is continuous, nonnegative and $\int_\Omega \tilde{k}(x,y) d \mu (y)=1.$ 
\end{flushleft}
\begin{flushleft}
	We consider now the integral operator $T_k$ on $C_{\mathbb{R} }(\Omega)$ given by
\end{flushleft}
\begin{center}
	$T_k (f) (x) =\int_\Omega \tilde{k} (x,y)f(y) d \mu (y) \ \ \textrm{for all } x \in \Omega .$
\end{center}
\begin{flushleft}
	It is straightforward to check that $T_k$ is a Markov operator with respect to $K$ and the constant function $1.$ The adjoint of $T_k,$ will be denoted by $ T_k^*$.\\
	Since $(C_\mathbb{R} (\Omega))^*=M_r(\Omega)$, where $M_r (\Omega )$ is the space of all signed Radon measures on $\Omega $, we have that $T_k^* $ is a bounded linear operator on  $M_r(\Omega).$\\
	
	\text{ }
	
	Now we are ready to present the main result in this section.
\end{flushleft}
\begin{proposition} \label{glavna} 
	Under the above assumptions, if $||\tilde{k}||_\infty < \frac{2}{\mu (\Omega )},$ then there exists a unique invariant probability Radon measure $ \tilde v $ on $ \Omega $ such that $$|(T_k^*)^n (v)-\tilde{v}| \ ( \Omega)\rightarrow 0 \mathrm{ \ as \ }  n\rightarrow \infty  $$ for all probability Radon measures $v$ on $ \Omega .$
\end{proposition}
\begin{proof}
	Let $ v \in M_r(\Omega)  $ and $w=T_k^*(v) .$ Then we have that

	\begin{center}
		$\int_\Omega f(y) dw (y)= \ \int_\Omega T_k(f) dv(x)$ $ \ = \ \int_\Omega ( \int_\Omega \tilde{k} (x,y)f(y) d \mu (y) )dv(x) .$
	\end{center}
	Suppose that $v=v_+-v_-$ is the Jordan decomposition of $v$. Then 
	$$\int_{\Omega}\Big(\int_{\Omega}|\tilde{k}(x,y)f(y)|d\mu(y)\Big)dv_+(x) \leq||\tilde{k}||_\infty \ ||f||_\infty \ \mu(\Omega) v_+(\Omega) < \infty.$$
	since $u,v_+$ are Radon measures and $\Omega $ is compact. Similarly, \\
	$$\int_{\Omega}\Big(\int_{\Omega}|\tilde{k}(x,y)f(y)|d\mu(y)\Big)dv_-(x) < \infty .$$
	By Fubini's theorem, we may change the order of the integration and get
	$$\int_{\Omega} \Big(\int_{\Omega}\tilde{k}(x,y)f(y)d\mu(y)\Big)dv(x) =\int_{\Omega} \Big(\int_{\Omega}\tilde{k}(x,y)f(y)dv(x)\Big))d\mu(y).$$
	\begin{flushleft}
		This gives that 
	\end{flushleft}
	\begin{center}
		$\int_\Omega f(y) dw (y) =\int_\Omega f(y) \rho (y) d\mu (y)$  where $\rho(y) =\int_\Omega \tilde{k} (x,y)  dv(x).$
	\end{center}
	\begin{flushleft}
		Since this holds for all $f \in C_\mathbb{R} (\Omega) ,$ we deduce that $dw=\rho d\mu .$ Then we get $d|w|=|\rho | d\mu $ since $\mu $ is a  positive measure. So we obtain that
	\end{flushleft}
	\begin{center}
		$|T_k^*(v)|(\Omega)=|w|(\Omega)=\int_\Omega |\rho (y)| d\mu (y)= \int_\Omega |\int_\Omega \tilde{k} (x,y)  dv(x)| d\mu (y)  .$
	\end{center}

	
	\begin{flushleft}
		Hence, using that
		$||\cdot ||_H^*=\frac{1}{2}||\cdot||_T^* $
		by definition and that $||v||_T^*=|v|(\Omega)$ for all $v \in M_r(\Omega) ,$ we derive 
	\end{flushleft}
	$$||S_k^*||_H^*=||T_{k_{|M(1)}}^* ||_H^* = \sup_{v \in M(1) , \ v \neq 0 } \frac{||T_K^* (v) ||_H^*}{||v||_H^*} $$
	\begin{center}
		$$=\sup_{v \in M(1) , \ v \neq 0 } \frac{||T_K^* (v) ||_T^*}{||v||_T^*} =\sup_{v \in M(1) , \ v \neq 0 } \frac{|T_K^* (v) |(\Omega)}{|v|(\Omega )} $$
	\end{center}
	\begin{center}
		$$=\sup_{v \in M(1) , \ v \neq 0 } \frac{\int_\Omega | \int_\Omega \tilde{k}(x,y) dv_+ (x) - \int_\Omega \tilde{k}(x,y) dv_- (x)| d\mu (y)}{v_+ (\Omega)  + v_- (\Omega)} .$$ 
	\end{center}
	
	Now, if $ v \in M(1) ,$ then $ v( \Omega) = 0 ,$ hence $v_+(\Omega)=v_-(\Omega)= \frac{1}{2}|v|(\Omega). $ So we get 
	$$ |\int_\Omega \tilde{k}(x,y) dv_+ (x)- \int_\Omega \tilde{k}(x,y) dv_- (x)| $$
	\begin{center}
		$\leq \max \lbrace \int_\Omega \tilde{k}(x,y) dv_+ (x), \  \int_\Omega \tilde{k}(x,y) dv_- (x) \rbrace  $ 
	\end{center}
	\begin{center}
		$\leq \max \lbrace \int_\Omega ||\tilde{k}||_\infty dv_+ , \  \int_\Omega ||\tilde{k}||_\infty dv_- \rbrace  $ 
	\end{center}
	\begin{center}
		$= \max \lbrace  ||\tilde{k}||_\infty v_+(\Omega ) , \  ||\tilde{k}||_\infty v_- (\Omega ) \rbrace  $ 
	\end{center}
	\begin{center}
		$=\frac{1}{2} ||\tilde{k}||_\infty |v|(\Omega )$ for all $y\in \Omega.$
	\end{center}
	In the first inequality we have used that $ \int_\Omega \tilde{k}(x,y) dv_+ (x), \  \int_\Omega \tilde{k}(x,y) dv_- (x) \ \geq 0 $ for all $ y \in \Omega $ since $ \tilde{k} \geq 0 .$ Hence we obtain that
	$$||T_{k_{|M(1)}}^* ||_H^*=\sup_{v \in M(1) , \ v \neq 0 } \frac{|T_K^* (v) |(\Omega)}{|v|(\Omega )}  $$
	\begin{center}
		$$ \leq \sup_{v \in M(1) , \ v \neq 0 } \frac{\int_\Omega \frac{1}{2} ||\tilde{k}||_\infty |v|(\Omega ) d\mu }{|v|(\Omega)}= \frac{1}{2}||\tilde{k}||_\infty \mu (\Omega ).$$
	\end{center}
	
	Therefore, if $\tilde{k} $ is such that $||\tilde{k}||_\infty < \frac{2}{\mu (\Omega )} ,$ then $|||\tilde{T}_k|||_H<1.$ By Theorem \ref{theomaster}, the proposition follows.
\end{proof}

\begin{flushleft}
	As a concrete example, let now $\Omega =[0,2 \pi ], \ \mu$ be the Lebesgue measure on $ [0, 2\pi]$ and $k:\Omega \ \times  \ \Omega \rightarrow  \mathbb{R} $ be given as $k(x,y)=\frac{1}{4} sin (\frac{1}{4}(x+y)).$ Then $k$ is continuous and $k(x+y) \geq 0 \ $ for all $ x,y \in [0,2 \pi ]$. Furthermore, we have that
\end{flushleft}
$$\int_\Omega k(x,y) d\mu (y)= \int_0^{2\pi} k(x,y)dy =cos(\frac{1}{4}x) + sin(\frac{1}{4}x) .$$
Using elementary calculus, it can be easily checked that 
$$\cos {(\frac{1}{4} x)}+\sin {(\frac{1}{4}x)}\geq 1 $$
for all $x \in [0,2\pi ].$ In particular, 
$$\cos {(\frac{1}{4}x)}+\sin {(\frac{1}{4}x)}> 0 $$
for all $x \in [0,2\pi ],$ so we can then define $\tilde{k}:\Omega \times \Omega \rightarrow \mathbb{R} $ as described above by letting 
$$\tilde{k}(x,y)=\frac{k(x,y)}{\int_{\Omega}k(x,y)d\mu (y)} = \frac{k(x,y)}{\cos{(\frac{1}{4}x)+\sin (\frac{1}{4}x)}} .$$ 
\begin{flushleft}
	Since $\cos(\frac{1}{4}x)+\sin(\frac{1}{4}x) \geq 1 $ for all $x \in [0,2\pi ],$ it follows that  $$||\tilde{k}||_\infty =\sup_{x,y \in [0,2\pi ]} \frac{\frac{1}{4}\mathrm{sin}(\frac{1}{4}(x+y))}{(\cos(\frac{1}{4}x)+\sin(\frac{1}{4}x ))}\leq \frac{1}{4},$$
\end{flushleft}
so $$||\tilde{k}||_\infty \leq \frac{1}{4} < \frac{1}{\pi }=\frac{2}{2\pi }=\frac{2}{\mu (\Omega )} ,$$ since $\Omega=[0,2\pi ] $ and $\mu $ is the Lebesgue measure.

\section{Spaceability of cones in the space of Radon measures }
\label{sec:4} 

In this section, for each $\mu,\nu\in M(\Omega) ,$ we say that $\mu\sim\nu$ if there are measurable mutually disjoint subsets $B_1,B_2$ of $\Omega$ such that 
$$|\mu|(B_1^c)=|\nu|(B_2^c)=0.$$


Now we recall the following definition.
\begin{definition}\label{dd}
	\cite[Definition 4.1]{MIA} Let $\mathcal{E}$ be a topological vector space. We say that a relation $\sim$ on $\mathcal{E}$ has property $(D)$ if the following conditions hold:
	\begin{enumerate}
		\item If $(x_n)$ is a sequence in $\mathcal{E}$ such that $x_n\sim x_m$ for all distinct index $m,n$, then for each disjoint finite subsets $A,B$ of $\mathbb{N}$ we have 
		$$\sum_{n\in A}\alpha_nx_n\sim \sum_{m\in B}\beta_m x_m,$$
		where $\alpha_n$ and $\beta_m$'s are arbitrary scalars.
		\item If a sequence $(x_n)$ converges to $x$ in $\mathcal{E}$ and for some $y\in \mathcal{E}$, $x_n\sim y$ for all $n\in\mathbb{N}$, then $x\sim y$.
	\end{enumerate}
\end{definition}

Then we present the following proposition.

\begin{proposition}
	The relation $\sim$ defined as above satisfies the conditions of Definition \ref{dd}.
\end{proposition}
\begin{proof}
	Let $\{\mu_n\}_n$ be a sequence in $M(\Omega)$ such that $\mu_n\sim\mu_m$ for all distinct indexes $m,n$. Let $A,B$ be finite mutually disjoint subsets of $\mathbb{N}$. For any fixed element $j\in A$ and each $m\in B$, there exist disjoint Borel subsets $D_{(j,m)}$ and $D'_{(j,m)}$ of $\Omega$ such that $|\mu_j|(D^c_{(j,m)})=|\mu_m|(D'^c_{(j,m)})=0$ and $D_{(j,m)} \cap D'_{(j,m)}=\varnothing$. It follows that 
	$$|\mu_j|(\bigcup_{m\in B}D_{(j,m)}^c)\leq \sum_{m\in B}|\mu_j|(D_{(j,m)}^c)=0.$$
	Hence, we get for all $i\in A$ that $|\mu_i|(\cap_{j\in A}\cup_{m\in B}D_{(j,m)}^c)=0$, so 
	\begin{align*}
		|\sum_{i\in A}\mu_i|(\cap_{j\in A}\cup_{m\in B}D_{(j,m)}^c)\leq \sum_{i\in A} \vert \mu_i \vert (  \cap_{j\in A}\cup_{m\in B}D_{(j,m)}^c)=0.
	\end{align*}
	Similarly, 
	$$\left|\sum_{i \in B}\mu_i\right|\left(\bigcap_{m\in B}\bigcup_{j\in A}D'^c_{(j,m)}\right)=0.$$
	Since we have that 
	$$(\cup_{j\in A}\cap_{m\in B}D_{(j,m)})\cap (\cup_{m\in B}\cap_{j\in A}D'_{(j,m)})=\varnothing,$$
	because $D_{(j,m)}\cap D_{(j,m)}'=\varnothing$ for all $j\in A$ and $m\in B$, 
	we conclude that $\sum_{i\in A}\mu_i\sim \sum_{m\in B}\mu_m$. By these arguments, it is not hard to see that we actually have $\sum_{i\in A} \alpha_i \mu_i\sim \sum_{m\in B} \beta_m \mu_m$ for all scalars $ \alpha_i , \beta_m $ with $ i \in A, m \in B ,$ since $| \alpha \mu|= | \alpha||\mu| $ for every scalar $ \alpha $ and every Radon measure $ \mu $ on $ \Omega .$ Thus, the condition $(1)$ of  Definition \ref{dd} is satisfied. \\
	
	Suppose now that $\{\mu_n\}_n$ is a sequence in $M(\Omega)$, and $\nu\in M(\Omega)$ such that $\mu_n\sim \nu$ for all $n$. Then, for each $n\in\mathbb{N}$ there exist measurable subsets $D_n$ and $D_n'$ such that $D_n\cap D_n'=\varnothing$ and $|\mu_n|(D_n^c)=|\nu|(D_n'^c)=0$.
	Hence, for every $m\in\mathbb{N}$ we have $|\nu|(\cup_{n=1}^mD_n'^c)=0$. This implies that $$|\nu|(\cup_{n=1}^\infty D_n'^c)=\lim_{m\rightarrow\infty}|\nu|(\cup_{n=1}^mD_n'^c)=0.$$
	Moreover, for each $n\in\mathbb{N}$ we have $|\mu_n|(\cap_{m=1}^\infty D_m^c)=0$. Since $D_n\cap D_n'=\varnothing$ for all $n$, we get 
	$$(\bigcup_{m=1}^\infty D_m)\cap (\cap_{n=1}^\infty D_n')=\varnothing.$$
	Now, if $\mu_n\rightarrow \mu$ as $n\rightarrow\infty$ for some $\mu\in M(\Omega)$, then 
	\begin{align*}
		|\mu|(\cap_{n=1}^\infty D_m^c)&\leq |\mu-\mu_n|(\cap_{m=1}^\infty D_m^c)+|\mu_n|(\cap_{m=1}^\infty D_m^c)\\
		&=|\mu-\mu_n|(\cap_{m=1}^\infty D_m^c)\rightarrow 0,
	\end{align*}
	as $n\rightarrow\infty$. Hence, 
	$$|\mu|(\cap_{n=1}^\infty D_m^c)=0.$$
	Thus, $\mu\sim\nu$, which proves the condition (2) in  Definition \ref{dd}.
\end{proof}
We recall that a subset $S$ of a Banach space $Y$ is called \emph{spaceable} in $Y$ if $S\cup\{0\}$ contains a closed  infinite-dimensional subspace of $Y$. In this section, a subset $B$ of a vector space is called a pseudocone if for each scalar $c,$ $cB \subseteq B.$ For a Borel measurable subset $E$ and some $\mu\in M(\Omega)$, we let $\mu_E$ be the measure given by $\mu_E(B):=\mu(B\cap E)$ for every Borel subset $B$ of $\Omega$. If $K$ is a pseudocone in $M(\Omega)$, we denote 
$$\widetilde{K}:=\{\mu_E:\,\mu\in K,\,E\text{ is Borel}\}.$$
Since for every scalar $\lambda\in\mathbb{C}$ we have $(\lambda \mu)_E=\lambda\,\mu_E$, it follows that $\widetilde{K}$ is a pseudocone since $K$ is a pseudocone. Moreover, $K\subseteq \widetilde{K}$.\\
Next we recall the following theorem.
\begin{theorem}\label{322}
	\cite[Theorem 4.2]{MIA}  Let $(\mathcal{E},\|\cdot\|)$ be a Banach space, $\sim$ be a relation on $\mathcal{E}$ with property $(D)$, and $K$ be a nonempty
	subset of $\mathcal{E}$. Assume that:
	\begin{enumerate}	
		\item there is a constant $k>0$ such that $\|x+y\|\geq k\,\|x\|$ for all $x,y\in \mathcal{E}$ with $x\sim y$;
		\item  $K$ is a pseudocone;
		\item if $x,y\in \mathcal{E}$ such that $x+y\in K$ and  $x\sim y$ then $x,y\in K$;
		\item there is an infinite sequence $\{x_n\}_{n=1}^\infty\subseteq \mathcal{E}\setminus K$ such that for each distinct $m,n\in\mathbb{N}$, $x_m\sim x_n$.
	\end{enumerate}
	Then, $\mathcal{E}\setminus K$ is spaceable in $\mathcal{E}$.
\end{theorem}
Then we give the following proposition.
\begin{proposition} \label{prop43}
	Let $K$ be a pseudocone in $M(\Omega)$. If there exists a sequence of mutually disjoint Borel subsets $\{E_n\}_{n\in\mathbb N}$ of $\Omega$ such that for all $n$ we have
	$$\{\mu_{E_n}:\,\mu\in\widetilde{K}\}\neq \{\mu_{E_n}:\, \mu\in M(\Omega)\},$$
	then $M(\Omega)\setminus\tilde{K}$ (and consequently $M(\Omega)\setminus K$) is spaceable in $M(\Omega)$. 
\end{proposition}
\begin{proof}
	We show that the relation $\sim$ satisfies the conditions of  Theorem \ref{322}. Let $\mu,\nu\in M(\Omega)$, and suppose that $\mu\sim\nu$. Then, there exist Borel subsets $B_1,B_2$ of $\Omega$ such that $B_1\cap B_2=\varnothing$ and $|\mu|(B_1^c)=|\nu|(B_2^c)=0$. Let $\{E_n\}_{n=1}^N$ be a partition of $\Omega$. Then, $\{E_n\cap B_1\}_{n=1}^N\bigcup \{B_1^c\}$ is also a partition of $\Omega$. Moreover, since $B_1\cap B_2=\varnothing$, we have $B_1\subseteq B_2^c$, hence 
	\begin{align*}
		\sum_{n=1}^N|(\mu+\nu)(E_n\cap B_1)|+|(\mu+\nu)(B_1^c)|&=\sum_{n=1}^N|\mu(E_n\cap B_1)|+|\nu(B_1^c)|\\
		&\geq \sum_{n=1}^N|\mu(E_n\cap B_1)|=\sum_{n=1}^N |\mu(E_n)|.
	\end{align*}
	In particular, $|\mu+\nu|(\Omega)\geq \sum_{n=1}^N |\mu(E_n)|$. Since $\{E_n\}_{n=1}^N$ was an arbitrary partition of $\Omega$, we deduce that $|\mu+\nu|(\Omega)\geq |\mu|(\Omega)$. So, the condition (1) in Theorem \ref{322} holds. \\
	
	Next, let $\mu , v \in M(\Omega) $ such that $\mu + v \in \tilde{K} $ and $\mu \sim v .$ Let $B_{1}, B_{2} $ be Borel subsets of $\Omega $ such that $B_{1} \cap B_{2} = \varnothing $ and $\vert \mu \vert (B_{1}^{c})= \vert v \vert (B_{2}^{c})=0  .$ Then, again since $B_{1} \subseteq B_{2}^{c} $ and $B_{2} \subseteq B_{1}^{c}  ,$ we get $(\mu + v)_{B_{1}} = \mu_{B_{1}} + v_{B_{1}}= \mu_{B_{1}}=\mu $ and, similarly, $(\mu + v)_{B_{2}} = v_{B_{2}}=v .$ However, since $\mu + v \in \tilde{K} ,$ by the definition of $ \tilde{K} $ there exists some measure $\eta \in K $ and some Borel subset $E$ of $ \Omega$ such that $\eta_{E}= \mu + v .$ Hence, $\mu=(\mu + v)_{B_{1}}= \eta_{E \cap B_{1}} \in \tilde{K} $ and $v=(\mu + v)_{B_{2}}= \eta_{E \cap B_{2}} \in \tilde{K}  .$ Thus, $ \mu, v \in \tilde{K},$ so the condition $(3)$ of Theorem \ref{322} is satisfied.\\
	
	By the assumptions in the proposition, the conditions (2) and (4) of Theorem \ref{322} are also satisfied, so the proposition follows.
\end{proof}

\begin{corollary}
	Let $K$ be the pseudocone of all scalar multiples of positive Radon measures on a locally compact Hausdorff space $\Omega $ with infinite cardinality. Then $ M(\Omega) \setminus K$ is spaceable in $  M(\Omega).$ Consequently, the complement of the cone of all positive Radon measures on $ \Omega $ is spaceable in  $  M(\Omega).$
\end{corollary}	

\begin{proof}
	Notice first that $K = \tilde{K} $ since $(\lambda \mu)_{E} = \lambda \mu_{E} $ for all $\lambda \in \mathbb{C}, $ $\mu \in M(\Omega)$ and any Borel subset $E.$ Since $ \Omega $ consists of infinitely many distinct elements by assumption,  we can find two mutually disjoint sequences $( x_{n} )_{n} $ and $( y_{n} )_{n} $ in $\Omega $ where $x_{n} \neq x_{m}$ and $y_{n} \neq y_{m}$ whenever $n \neq m.$ \\
	For each $ n \in \mathbb{N},$ set $\mu_{n} = \delta_{x_{n}} + i \delta_{y_{n}} $ where $\delta$ denotes the Dirac measure. It is not hard to see that $ \mu_{n} \notin M(\Omega) \setminus K$ for all $n \in \mathbb{N} ,$ so the sequence $ ( \mu_{n} )_{n} $ satisfies the condition (4) of Theorem \ref{322}.  Since $K=\tilde{K} ,$ setting $E_{n}= \lbrace x_{n} \rbrace \cup \lbrace y_{n} \rbrace $ for each $n \in \mathbb{N}$  we can apply Proposition \ref{prop43} to deduce that $ M(\Omega) \setminus K $ is spaceable in $ M(\Omega) .$
\end{proof}


\bibliographystyle{amsplain}

\end{document}